\documentclass[12pt,twoside,reqno]{amsart}
\usepackage{latexsym,amsmath,amsopn,amssymb,amsthm,amsfonts}
\usepackage[T2A]{fontenc}
\usepackage[cp1251]{inputenc}
\usepackage[russian,english]{babel}
\usepackage[backref=page, breaklinks=true,colorlinks=true,linkcolor=blue,citecolor=blue,urlcolor=blue]{hyperref}

\usepackage{graphicx}

\DeclareMathOperator{\ad}{ad}
\DeclareMathOperator{\Id}{Id}

\DeclareMathOperator{\diag}{diag}

\DeclareMathOperator{\Ad}{Ad}

\DeclareMathOperator{\trace}{trace}

\DeclareMathOperator{\Ric}{Ric}

\newtheorem{theorem}{Theorem}
\newtheorem{prop}{Proposition}
\newtheorem{lemma}{Lemma}
\newtheorem{corollary}{Corollary}

\newtheorem{problem}{Problem}

\theoremstyle{definition}
\newtheorem{definition}{Definition}
\newtheorem{remark}{Remark}
\newtheorem{example}{Example}

\begin{document}
\title[]{Homogeneous spaces with geodesic orbit Riemannian metrics and with integrable invariant distributions}
\author{V.\,N.\,Berestovskii, Yu.\,G.\, Nikonorov}

\address{Berestovskii Valerii Nikolaevich \newline
Sobolev Institute of Mathematics SB RAS,
\newline 4 Acad. Koptyug Ave, Novosibirsk, 630090, Russia}
\email{vberestov@inbox.ru}

\address{Nikonorov Yurii Gennadievich \newline
Southern Mathematical Institute of VSC RAS,  \newline
53 Vatutina St., Vladikavkaz, 362025, Russia}
\email{nikonorov2006@mail.ru}

\begin{abstract}
We consider  two types of homogeneous spaces of Lie groups with compact stabilizer subgroups:
those with integrable invariant distributions and those for which all invariant Riemannian metrics are geodesic orbit.
Recall that a Riemannian metric is called geodesic orbit if each geodesic is an orbit of a one-parameter subgroup of the full isometry group.
We exhibit several homogeneous spaces of the first type that do not belong to the second one.

\vspace{2mm}

\vspace{2mm} \noindent {\it Key words and phrases}: Einstein space, geodesic orbit manifold, geodesic orbit space, homogeneous geodesic,
homogeneous space, integrable invariant distributions, rigid GO-space,
solvmanifold, strong subalgebra, symmetric space

\end{abstract}

\maketitle
%\tableofcontents

\section*{Introduction}

It was stated in \cite{Ber1989} and proved in \cite{Ber1989a} that a homogeneous space $G/H$ of a connected Lie group $G$ modulo a compact subgroup
$H$ admits no invariant Carnot~--~Carath\'eodory (sub-Riemannian) metric if and only if $G/H$ is a homogeneous space with integrable invariant distributions.
The latter means that every invariant distribution on $G/H$ is integrable (involutive).

As the first examples of such spaces, \cite{Ber1989} indicated symmetric spaces of \'E.~Cartan \cite{Cart1926}, isotropy irreducible spaces,
and Lie groups $G$ (with trivial $H$) of dimension $\ge 2$ such that every left-invariant Riemannian metric on $G$ is flat (if $G$ is abelian)
or has constant negative sectional curvature \cite{Milnor1976}.

Earlier, O.~V.~Manturov in \cite{Mant1966, Mant1966e} (1966), J.~A.~Wolf in \cite{Wolf1968} (1968), and M.~Kr\"amer in~\cite{Kram} (1975)
classified strongly isotropy irreducible spaces (when $H$ is connected).
Later, M.~Wang and W.~Ziller classified all isotropy irreducible spaces in \cite{WanZil1991}; see also \cite{WanZil1993}.

It is proved in \cite{Ber1992, Ber1995} that every compact simply connected homogeneous space with integrable invariant distributions is a direct product
of compact simply connected (strongly) isotropy irreducible spaces.

Some other homogeneous spaces with integrable invariant distributions were found in \cite{Gor2008} and \cite{BerGor2014}. Nevertheless,
the problem of classifying homogeneous spaces with integrable invariant distributions, posed in \cite{Ber1989}, has not been solved.
This problem appears to be very difficult, as the partial but nontrivial examples mentioned above demonstrate.

It is known that every invariant Riemannian metric $\mu$ on any homogeneous space $M$ mentioned in the second paragraph above
is geodesic orbit, i.e., every geodesic in $(M,\mu)$ is the orbit of a one-parameter subgroup of the isometry group of $(M,\mu)$.
In \cite{BerGor2014} and \cite{BerNik2020}, the question was posed whether every invariant Riemannian metric on any homogeneous space
with integrable invariant distributions is geodesic orbit.

The main result of this paper is a negative answer to the latter question in the general case.

\section{Homogeneous spaces with integrable invariant distributions: some known results}

In what follows, $(M,(\cdot,\cdot))$ denotes a connected complete Riemannian manifold of class $C^{\infty}$
with the inner product $(\cdot,\cdot)$ defined on the tangent bundle $TM$;
$I(M,(\cdot,\cdot))$ is the full isometry group of $(M,(\cdot,\cdot))$;
$M=G/H$ is the homogeneous space of a connected Lie group $G$ with respect to its compact subgroup $H$, $o =eH\in M$.

\begin{remark}
It is assumed that $G/H$ is effective, i.e. $H$ does not contain non-trivial normal subgroups of $G$.
\end{remark}

\begin{definition}
A homogeneous space with integrable (involutive) invariant distributions is a space $G/H$ such that every $G$-invariant distribution on $G/H$ is integrable (involutive).
\end{definition}

One can use the property A) from the following result instead of the original  definition of homogeneous spaces with involutive invariant distributions.

\begin{prop}
[\cite{Ber1989}]\label{prop1}
Let $(\mathfrak{g}, [\cdot, \cdot])$ be the Lie algebra of the Lie group $G$, $\mathfrak{h}\subset \mathfrak{g}$
be the Lie algebra of the Lie subgroup $H$. Then $G/H$ is a homogeneous manifold with integrable invariant distributions if and only if
{\rm\bf A)} every vector subspace
of $\mathfrak{q}\subset \mathfrak{g}$ such that $\mathfrak{h}\subset \mathfrak{q}$ and $\Ad(H)(\mathfrak{q})\subset \mathfrak{q}$, is a Lie subalgebra of
$(\mathfrak{g}, [\cdot, \cdot])$, i.e.
$[\mathfrak{q},\mathfrak{q}]\subset \mathfrak{q}$.
\end{prop}

\begin{corollary}[\cite{Ber1989}] \label{cor1}
If $H$ is a connected Lie group, in particular if $G/H$ is simply connected, then $G/H$ is a homogeneous manifold with integrable invariant distributions
if and only if every vector subspace of $\mathfrak{q}\subset \mathfrak{g}$ such that $\mathfrak{h}\subset \mathfrak{q}$ and
$[\mathfrak{h},\mathfrak{q}]\subset \mathfrak{q},$ is a Lie subalgebra of $(\mathfrak{g}, [\cdot, \cdot])$, i.e. $[\mathfrak{q},\mathfrak{q}]\subset \mathfrak{q}$.
\end{corollary}

According to Corollary \ref{cor1}, Lie groups $G$
correspond to pairs $(\mathfrak{g}, \mathfrak{h}=0)$ such that any vector subspace of $\mathfrak{g}$ is a subalgebra of  the Lie algebra $\mathfrak{g}$.
All such pairs $(\mathfrak{g}, \mathfrak{h})$ were classified in \cite{Milnor1976}. These are either commutative Lie algebras $\mathfrak{g}$
or the Lie algebra $\mathfrak{g}_n,$ $n\geq 2,$ of the Lie group $G_n$, generated by the
parallel translataions and homotheties of $\mathbb{R}^{n-1}$
(see ``Special Example 1.7'', pp. 298--299 in \cite{Milnor1976}).
Every Lie group $G_n$, $n\geq 2,$ is simply connected, and any left-invariant Riemannian metric on $G_n$ has constant negative sectional curvature \cite{Milnor1976},
hence is isometric to the Lobachevsky space, which is a Riemannian symmetric space.

\begin{prop}
[\cite{Ber1989}]
\label{iigo}
Riemannian symmetric spaces {\rm(\cite{Cart1926})} and isotropy irreducible homogeneous spaces {\rm(\cite{Mant1966, Wolf1968, Kram, WanZil1991, WanZil1993})}
are homogeneous spaces with integrable invariant distributions.
\end{prop}

More detailes on mentioned results from \cite{Ber1989} are given in \cite{Ber1989a}.

The following result is almost obvious.

\begin{prop}
All invariant distributions on the direct product of isotropy irreducible homogeneous spaces are integrable.
\end{prop}

\begin{remark}
Any product of more than one of the above-mentioned Lie groups with integrable
invariant distributions, at least one of which is noncommutative, has a nonintegrable
invariant distribution.
\end{remark}

\begin{theorem}[\cite{Ber1992, Ber1995}]
\label{th_comprodirr}
A simply connected compact homogeneous space with integrable invariant distributions is isomorphic to a direct product of simply connected compact
isotropy irreducible homogeneous spaces.
\end{theorem}

In  papers \cite{Gor2008,BerGor2014}, new homogeneous spaces with integrable invariant
distributions were found, in particular, those that are not direct
products of isotropy irreducible homogeneous spaces.

{\it Further, we will consider homogeneous spaces $G/H$ with connected Lie groups $G$ and $H$ {\rm(}$H$ is compact{\rm)}.}

\begin{definition}
[\cite{Gor2008, BerGor2014}]
A subalgebra $\mathfrak{h}$ of a Lie algebra $\mathfrak{g}$ is called {\it strong} if (a) it satisfies the conditions of Corollary \ref{cor1} and
(b) the corresponding connected subgroup $H$ (of a connected Lie group $G$ with Lie algebra $\mathfrak{g}$) is compact.
\end{definition}

Let $\mathfrak{g}$ be an arbitrary Lie algebra that includes a strong proper subalgebra $\mathfrak{h}$.
The Lie algebra $\mathfrak{h}$ is compact, so it is always possible to choose a Levi decomposition $\mathfrak{g}=\mathfrak{s}+\mathfrak{r}$
($\mathfrak{r}$ is the radical and $\mathfrak{s}$ is the semisimple part) of the Lie algebra $\mathfrak{g}$
such that
$[\mathfrak{h},\mathfrak{s}]\subset \mathfrak{s}$.
Denote by $\mathfrak{n}$ the nilradical of the Lie algebra $\mathfrak{g}$.
A suitable Levi decomposition of
$\mathfrak{g}$ yields a decomposition $\mathfrak{h}=\mathfrak{h}_s+\mathfrak{h}_r$,
where $\mathfrak{h}_r=\mathfrak{h}\cap\mathfrak{r}$ is a torus in $\mathfrak{r}$
(an abelian compact Lie subalgebra), and $\mathfrak{h}_s=\mathfrak{h}\cap\mathfrak{s}$ is a compact Lie algebra (not necessarily semisimple).

\begin{theorem}
[\cite{Gor2008, BerGor2014}]\label{th.str1}
Let  $\mathfrak{h}$ be a strong subalgebra in a Lie algebra $\mathfrak{g}=\mathfrak{s}+\mathfrak{r}$
and let the pair $(\mathfrak{g},\mathfrak{h})$ be effective.
Then

{\rm(i)} The subalgebra $\mathfrak{h}_n=\mathfrak{h}\cap\mathfrak{s}_n$ is a maximal compact subalgebra in the non-compact ideal $\mathfrak{s}_n$
of a suitable semisimple part $\mathfrak{s}$ of the Lie algebra $\mathfrak{g}$.

{\rm(ii)} The subalgebra $\mathfrak{h}_c=\mathfrak{h}\cap\mathfrak{s}_c$ is a strong subalgebra of the corresponding compact semisimple Lie algebra $\mathfrak{s}_c$.

{\rm(iii)} The subalgebra $\mathfrak{h}_s$ is a direct sum of $\mathfrak{h}_n$ and $\mathfrak{h}_c$.

{\rm(iv)} The radical $\mathfrak{r}$ of the Lie algebra $\mathfrak{g}$ decomposes into a semidirect sum $\mathfrak{a}\ltimes\mathfrak{n}$
of some abelian subalgebra $\mathfrak{a}$
{\rm(}containing $\mathfrak{h}\cap\mathfrak{r}${\rm)} and the abelian nilradical $\mathfrak{n}$.
\end{theorem}

Let us prove some additional structure result.

\begin{prop}\label{pr.radstr1}
In the conditions of Theorem \ref{th.str1}, let $\mathfrak{n}^{\perp}$ be the orthogonal complement to the nilradical $\mathfrak{n}\subset \mathfrak{r}$ of $\mathfrak{g}$
in the radical
$\mathfrak{r}$ relative to an $\ad(\mathfrak{h})$-invariant inner product $(\cdot,\cdot)$ on $\mathfrak{g}$. Then $[\mathfrak{n}^{\perp},\mathfrak{s}]=0$.
\end{prop}

\begin{proof} Since $\mathfrak{s}$ is chosen to be $\ad(\mathfrak{h})$-invariant, $\mathfrak{n}$ and $\mathfrak{n}^{\perp}$ are $\ad(\mathfrak{h})$-invariant, then
$\mathfrak{n}^{\perp}\oplus \mathfrak{s}$ is also $\ad(\mathfrak{h})$-invariant. Since $\mathfrak{h}$ be a strong proper Lie subalgebra of $\mathfrak{g}$,
then $\mathfrak{n}^{\perp}\oplus \mathfrak{s} +\mathfrak{h}$ is a subalgebra of $\mathfrak{g}$. On the other hand, it is well known that
$[\mathfrak{s},\mathfrak{n}^{\perp}]\subset [ \mathfrak{g},\mathfrak{r}]\subset \mathfrak{n}$. Therefore, $[\mathfrak{s},\mathfrak{n}^{\perp}]=0$
(since $\mathfrak{h}\cap \mathfrak{n}=0$), q.e.d.
\end{proof}

By Theorem \ref{th.str1}, the radical $R$ of any Lie group $G,$ containing some strong subgroup
$H\neq \{e\},$ has a form $R=A\ltimes B,$ semidirect product of abelian Lie groups $A,$ $B.$

In particular, {\it if $G$ is solvable itself}, then it also has such structure.

\medskip

\begin{example}
\label{ex}
Let us consider the group $G=R=A\ltimes B,$ $B=(\mathbb{R}^n,+),$
$A=\mathbb{R}^q\times \operatorname{SO}(2)^l,$ $l,q\in \{0\}\cup\mathbb{N},$ $l+q\geq 1,$ $H=\operatorname{SO}(2)^l.$
The Lie group $G,$ as a semidirect
product, is specified by a {\it monomorphism} $\varphi: A\rightarrow \operatorname{GL}(n,\mathbb{R}).$

Assume that $A^{\#}:=\varphi(A)\subset \operatorname{GL}(n,\mathbb{R}),$ $H^{\#}:=\varphi(H),$ and any $H^{\#}$-invariant vector subspace $W\subset \mathbb{R}^n$ is also $A^{\#}$-invariant.
\end{example}

Using the proof of Proposition 5 in \cite{BerGor2014}, we obtain that $A^{\#}$ is a completely reducible group.
Then, the Proposition 6 of \cite{BerGor2014} implies the following statement.

\begin{prop}
\label{alga}
In some basis of the space $\mathbb{R}^n,$ the Lie algebra $a^{\#}$ of $A^{\#}$ has a form
\begin{equation}
\label{pa}
a^{\#}\subset \oplus_{i=1}^{l}\left(\begin{array}{cc}
a_i & -b_i\\
b_i & a_i
\end{array}\right)\oplus \{a_+\,E_{n-2l}\},
\end{equation}
where $0\leq 2l\leq n,$ $0\leq q\leq l+\min\{1,n-2l\}$, $E_{n-2l}$ is the identity matrix of order $n-2l$ if $n>2l$;
$b_i$,  $a_i, a_+$ are real-valued variables with linearly independent $b_i,$ $i=1,\dots, l$.
\end{prop}

\begin{remark}
The equality $l=0$ is equivalent to the condition that $H$ is trivial.
\end{remark}

\section{Geodesic orbit manifolds and the main problem}
\vspace{0.2cm}

\begin{definition}
A geodesic $\gamma$ on a Riemannian manifold $(M, (\cdot,\cdot))$ is called homogeneous
if it is an orbit of a one-parameter subgroup $\overline{\gamma}(t)\subset I(M,(\cdot,\cdot))$
(and, therefore, an integral curve of the Killing vector field generated by this subgroup).
\end{definition}

\begin{definition}
A Riemannian manifold $(M, (\cdot,\cdot))$ is called a manifold with homogeneous geodesics or a geodesic orbit manifold
(briefly, a GO-manifold) if each of its geodesics is homogeneous.
\end{definition}

It is clear that such a manifold is {\it homogeneous}, i.e. it admits a transitive isometry group, for example, the group $I(M, (\cdot,\cdot))$
or its connected component of the identity.

Famous examples of geodesic orbit manifolds are {\it Riemannian symmetric spaces}.

\begin{definition}
\label{ex6}
A homogeneous Riemannian space $(M=G/H, (\cdot,\cdot))$ is called a geodesic orbit (GO-space) if each its geodesic
is an orbit of a one-parameter subgroup in $G$.
\end{definition}

 This terminology was introduced in
\cite{KV} by O.~Kowalski and L.~Vanhecke, who initiated a systematic study of such spaces.
In the same paper, O.~Kowalski and L.~Vanhecke classified all geodesic orbit Riemannian manifolds
of dimension $\leq 6$.
We refer to \cite{KV}, \cite{Gor96}, \cite{Arv17}, \cite{Nik2017}, \cite{GorNik2018}, \cite{NikZilW25}  (and to the references
therein)
for expositions on general properties of geodesic
orbit Riemannian manifolds, classifications of some of  their subclasses and historical surveys.
\smallskip

Geodesic orbit manifolds and their realizations as homogeneous
spaces $M=G/H$ with $G$-invariant Riemannian metrics $(\cdot,\cdot)$ are the main goal of the studies in the monograph \cite{BerNik2020}.

Since the group $H$ is compact, the homogeneous space $G/H$ {\it is reductive}, that is, there exists a vector subspace $\mathfrak{p}\subset \mathfrak{g}$ such that
$$
\mathfrak{g}=\mathfrak{h}\oplus\mathfrak{p},\quad {\Ad}_{H}(\mathfrak{p})\subset \mathfrak{p}
\quad \Rightarrow \quad
[\mathfrak{h},\mathfrak{h}]\subset\mathfrak{h},\quad [\mathfrak{h},\mathfrak{p}]\subset\mathfrak{p}.
$$

We define the operator $U:\mathfrak{p}\times\mathfrak{p}\rightarrow\mathfrak{p}$ by the formula
$$
2(U(X,Y),Z)=([Z,X]_{\mathfrak{p}},Y)+(X,[Z,Y]_{\mathfrak{p}}).
$$

\begin{theorem}[\cite{BerNik2020}]\label{th_main}
Let $(M=G/H, g)$ be a homogeneous Riemannian manifold, $X\in\mathfrak{p}$, $Y\in\mathfrak{h}$.
Then $\gamma(t)=\exp(t(X+Y))(o)$ is a geodesic if and only if one of the following conditions holds:

{\rm 1)} $[X,Y]=U(X,X)$;

{\rm 2)} $([Y,X],Z)=(X,[X,Z]_{\mathfrak{p}})$, for any $Z\in\mathfrak{p}$;

{\rm 3)} $([X+Y,Z]_{\mathfrak{p}}, X)=0$ for any $Z\in\mathfrak{p}$.

A homogeneous Riemannian manifold $(M=G/H, g)$ is a geodesic orbit space if and only if for any $X\in\mathfrak{p}$ there exists
$Y\in\mathfrak{h}$ such that $X+Y$ is a geodesic vector, i.e. $({[X+Y,Z]}_{\mathfrak{p}}, X)=0$ for all $Z\in\mathfrak{p}$.
\end{theorem}

We need some extensive citations from \cite{Gord1981}, \cite{Gor96}, and \cite{GorNik2018}. First, we introduce the necessary notation.

Given connected Lie groups $A$ and $G$ with $G\subset A$ choose Levi factors $G_{ss}$ and $A_{ss}$ of $G$ and $A$ with $G_{ss}\subset A_{ss}$.
Denote by $\mathfrak{a}$, $\mathfrak{g}$, $\mathfrak{a}_{ss}$, and $\mathfrak{g}_{ss}$ the Lie algebras of $A$, $G$, $A_{ss}$, and $G_{ss}$ respectively.
Write $\mathfrak{a}_{ss}=\mathfrak{a}_{nc}\oplus\mathfrak{a}_c$ and
$\mathfrak{g}_{ss}=\mathfrak{g}_{nc}\oplus\mathfrak{g}_c$, where
$\mathfrak{a}_{nc}$ and $\mathfrak{g}_{nc}$  are semisimple Lie algebras of the noncompact type,
while $\mathfrak{a}_{c}$ and $\mathfrak{g}_{c}$ are of the compact type.
Let $A_{nc}$, $A_{c}$, $G_{nc}$, and $G_c$ be the connected subgroups in $A$ with Lie algebras $\mathfrak{a}_{nc}$, $\mathfrak{a}_{c}$,
$\mathfrak{g}_{nc}$, and $\mathfrak{g}_{c}$.

\begin{theorem}[Theorem (2.2) in \cite{Gord1981}]
\label{th_nc}
Let the connected Lie group $A$ be a product $A=GL$ of a connected subgroup $G$ with nilpotent radical and a compact subgroup~$L$. Then $A_{nc}=G_{nc}$.
\end{theorem}

\begin{theorem}[Theorem (3.1) in \cite{Gord1981}]
\label{th_rad}
Let $A=GL$ be as in Theorem \ref{th_nc} and suppose the radical of $G$ is nilpotent. Denote the radicals of $\mathfrak{a}$ and $\mathfrak{g}$ by
$\mathfrak{m}$ and $\mathfrak{n}$ respectively. Then

{\rm(a)} $\mathfrak{n}$ is the sum of ideals $\mathfrak{n}=\mathfrak{n}_1\oplus\mathfrak{n}_2$ where
$\mathfrak{n}_1=\mathfrak{n}\cap\mathfrak{a}_{ss}$ is central in $\mathfrak{g}$ and $[\mathfrak{g},\mathfrak{n}]\subset \mathfrak{n}_2$.

{\rm(b)} $\mathfrak{m}$ is a vector space direct sum $\mathfrak{m}=\mathfrak{u}\oplus\mathfrak{n}'_2$ of an abelian subalgebra
$\mathfrak{u}$, compactly embedded in $\mathfrak{a}$, and an ideal $\mathfrak{n}'_2$ containing $[\mathfrak{g},\mathfrak{n}]$.

{\rm(c)} $[\mathfrak{a},\mathfrak{m}]\subset\mathfrak{n}'_2$ and $[\mathfrak{g}_{ss},\mathfrak{m}]=[\mathfrak{g}_{ss},\mathfrak{n}]$.

{\rm(d)} There exists an isomorphism $\varphi:\mathfrak{g}_{ss}+\mathfrak{n}_1+\mathfrak{n}'_2\rightarrow \mathfrak{g}$ which maps $\mathfrak{n}'_2$ onto
$\mathfrak{n}_2$ and restricts to the identity map on $[\mathfrak{g},\mathfrak{g}]+\mathfrak{n}_1$.
\end{theorem}

\begin{prop}[\cite{GorNik2018}]
\label{gornik}
Let $G/H$ be a connected Riemannian $GO$-space and let $\operatorname{Lev}(G)$ be any Levi factor of $G$.
Then the noncompact part $\operatorname{Lev}(G)_{nc}$ is a normal subgroup of $G$, i.e. $\operatorname{Lev}(G)_{nc}$ commutes with the radical
$\operatorname{Rad}(G)$ of $G$.
\end{prop}

We need also the following result by C.~Gordon (see also \cite{Nik2013} for more simple proof).

\begin{theorem}[Theorem 5.1 in \cite{Gor96}]\label{gordonnonpos}
Every Riemannian GO-manifold of nonpositive Ricci curvature is symmetric.
\end{theorem}

The homogeneous space $G/H$ of an arbitrary connected Lie group $G$ with compact
stabilizer $H$ will be called a {\it rigid GO-space}, if a homogeneous Riemannian manifold
$(G/H,\mu)$ is isometric to some GO-space for every $G$-invariant Riemannian metric $\mu$
on $G/H$. It is equivalent to the condition that $(G/H,\mu)$ is GO-manifold for every $G$-invariant Riemannian metric $\mu$ on $G/H.$
\medskip

It is well known that any direct product of isotropy irreducible spaces is a rigid GO-space.
Another important subclass of rigid GO-spaces is formed by all {\it weakly-symmetric spaces} (see details e.g. in \cite{AV}, \cite{Selberg}, \cite{Zil1996}).
In particular, any homogeneous space $SO(2n+1)/U(n),$ $n\geq 3,$ or $Sp(n)/U(1)\cdot Sp(n-1)=\mathbb{C}P^{2n-1}$, $n\geq 2$, is a rigid GO-space
(see details e.g. in \cite{AlNik}), but it is not a product of isotropy irreducible spaces and thus {\it it is not a homogeneous space
with integrable invariant distributions}.

It is interesting that $SO(2n+1)/U(n)$ ($Sp(n)/U(1)\cdot Sp(n-1))$),  are diffeomorphic to the symmetric spaces $SO(2(n+1))/U(n+1)$,
the spaces of orthogonal complex structures on $\mathbb{R}^{2(n+1)}$, (resp., $SU(2n)/S(U(1)\times U(2n-1))$), \cite{BerNik2020} and references there.
\smallskip

\begin{problem}[Problem 10 in \cite{BerGor2014}]
\label{conj}
Determine whether any  homogeneous
space $G/H$ of a connected Lie group $G$ with compact stabilizer subgroup $H$
and integrable invariant distributions is a rigid GO-space.
\end{problem}

The problem of classifying homogeneous spaces with integrable invariant distributions posed in \cite{Ber1989}, has not been solved.
Thus, Problem \ref{conj} is an attempt to generalize known results. Note also that the problem of classification of all rigid GO-spaces is very natural
(see e.g. \cite[Question 4]{Nik2017} and  \cite[Question 5.12.25]{BerNik2020}).

We have the following corollary from Theorem \ref{th_comprodirr}.

\begin{corollary}
Any compact simply connected homogeneous space  with integrable invariant distributions is a rigid GO-space.
\end{corollary}

The statements before Proposition \ref{iigo} and Definition \ref{ex6} imply

\begin{theorem}[\cite{Ber1989}]
Any Lie group $G$ {\rm(}with $H=\{e\}${\rm)} with integrable
{\rm(}left-{\rm)}invariant distributions is a rigid GO-space.
\end{theorem}

\section{Some examples and counterexamples}

Here we consider some examples of homogeneous spaces with integrable invariant distributions and strong subalgebras.

\begin{example}[\cite{BerGor2014}]\label{ex_0}
\label{ex_1}
Let  $G= H\ltimes \mathbb{R}^n$ be a semidirect product of a nontrivial compact Lie group $H$ and $(\mathbb{R}^n,+)$ (corresponding to the embedding
$H\subset  \operatorname{GL}(n,\mathbb{R}))$. Then the subalgebra $\mathfrak{h}$ is a strong subalgebra in the Lie algebra $\mathfrak{g}$ of the Lie group
$G$, $\mathbb{R}^n$ is the nilradiacal of $G.$ If $H$ is commutative, then $G$ is solvable Lie group.
\end{example}

This is Example 2 on p. 320 in \cite{BerGor2014} (Example 1 in \cite{BerGor2014} when $H=T$ is a torus).

It is easy to see that {\it the space $G/H$ is a rigid GO-space}, because {\it every invariant Riemannian metric $\mu$ on $G/H=\mathbb{R}^n$ is Euclidean,
and so is geodesic orbit}.

\begin{example}
\label{ex_2}
Let $\mathfrak{s}$ be a semisimple Lie algebra of noncompact type, $\mathfrak{h}$ its maximal compact subalgebra.
Suppose that there is a representation $\psi: \mathfrak{s} \mapsto \operatorname{\mathfrak{gl}(n,\mathbb{R})}$ such that any subspace in $\mathbb{R}^n$ invariant under
the action of $\psi(\mathfrak{h})$ is also invariant under $\psi(\mathfrak{s})$. Then the Lie algebra $\mathfrak{h}$ is strong in the semidirect sum
$\mathfrak{g}=\mathfrak{s}\ltimes\mathbb{R}^n $ (here $\mathbb{R}^n$ is the nilradical of $\mathfrak{g}$ and it is abelian),
which is defined by the representation $\psi$.

Indeed, let $\mathfrak{p}$ be a complement to
$\mathfrak{h}$ in $\mathfrak{s}$ with respect to the Killing form of $\mathfrak{s}$. It is well known that $[\mathfrak{h},\mathfrak{p}]\subset \mathfrak{p}$ and
$[\mathfrak{p},\mathfrak{p}]=\mathfrak{h}$. Therefore, $\mathfrak{p}\oplus\mathbb{R}^n$ is a complement to $\mathfrak{h}$ in $\mathfrak{g}$ and
any $\mathfrak{h}$-invariant subspace in $\mathfrak{p}\oplus\mathbb{R}^n$ has a form $\mathfrak{q}_1\oplus \mathfrak{q}_2$, where $\mathfrak{q}_1 \subset \mathfrak{p}$ and
$\mathfrak{q}_2\subset \mathbb{R}^n$  are respectively $\ad(\mathfrak{h})$-invariant and $\psi(\mathfrak{h})$-invariant. We know that $\mathfrak{h}\oplus \mathfrak{q}_1$ is a subalgebra in $\mathfrak{s}$ and $\mathfrak{q}_2$ is $\ad(\mathfrak{s})$-invariant.
Then $\mathfrak{h}\oplus\mathfrak{q}_1\oplus \mathfrak{q}_2$ is a subalgebra in $\mathfrak{g}$.
\end{example}

In fact, this construction were considered in \cite[\S 3]{Gor2008}.

\begin{example}
\label{ex_4}
Let us consider a partial case of Example \ref{ex_2}. We suppose in addition that
$\mathfrak{s}$ is a simple noncompact Lie algebra and $\psi(\mathfrak{s})$ acts irreducibly on $\mathbb{R}^n$.
We take $\mathfrak{q}_1= \mathfrak{p}$ and  $\mathfrak{q}_2=  \mathbb{R}^n$.
Since $\ad(\mathfrak{h})$ acts irreducibly both on $\mathfrak{q}_1$ and $\mathfrak{q}_2$, the space of  $\ad(\mathfrak{h})$-invariant inner products
on $\mathfrak{q}:=\mathfrak{q}_1\oplus\mathfrak{q}_2$ is two-parametric.
We fix an $\ad(\mathfrak{h})$-invariant inner product $(\cdot,\cdot)=(\cdot,\cdot)_1+(\cdot,\cdot)_2$, where
$(\cdot,\cdot)_1$ and $(\cdot,\cdot)_2$ are inner product on $\mathfrak{q}_1$ and $\mathfrak{q}_2$ respectively.
\end{example}

\begin{theorem}
\label{ce1}
Let $G/H$ be a homogeneous space with connected Lie groups $G$ and $H$ such that $G=S\ltimes\mathbb{R}^n$ and Lie algebras
$\mathfrak{g}=\mathfrak{s}\ltimes\mathbb{R}^n,$ $\mathfrak{h}$ as in Example \ref{ex_4} with $\psi\neq 0$.
Then $G/H$ is a homogeneous space with integrable invariant distributions but $G/H$ admits no invariant geodesic orbit Riemannian metric.
\end{theorem}

\begin{proof}
The Lie group $G$ has the commutative (hence, nilpotent) radical $\mathfrak{n}=\mathbb{R}^n$.
The conditions imply that $[\mathfrak{g},\mathfrak{n}]=\mathfrak{n}$ and $[\mathfrak{g},\mathfrak{g}]=\mathfrak{g}$.

Assume that $G/H$ admits an invariant geodesic orbit Riemannian metric $\mu$. Then there exists a connected Lie group $A=GL$ with a compact subgroup
$L$ such that $(G/H,\mu)$ is isometric to the geodesic orbit Riemannian space $(A/(HL),\nu)$.
Then $A_{nc}=G_{nc}=S$ by Theorem \ref{th_nc}.

The statements in the second line of the proof and Theorem \ref{th_rad} imply  that $\mathfrak{n}=\mathfrak{n}_2$, $\mathfrak{n}_1=\{0\}$,
$\mathfrak{n}=\mathfrak{n}'_2$ and $\mathfrak{n}\subset \mathfrak{m}$.

Now by Proposition \ref{gornik}, $A_{nc}=S$ commutes with $\operatorname{Rad}(A)\supset \operatorname{Rad}(G)=\mathbb{R}^n$.
This contradicts to the fact that $\psi(\mathfrak{s})$ acts irreducibly on $\mathfrak{n}$.
\end{proof}

\begin{corollary}
The statements of Theorem \ref{ce1} are true for the homogeneous space $G/H,$ where $G=SL(n,\mathbb{R})\ltimes \mathbb{R}^n,$ $H=SO(n)$ and $n\geq 3$.
\end{corollary}

\begin{remark}
The statements of Theorem \ref{ce1} and its proof are valid for noncompact semisimple connected Lie group $S$ with Lie algebra $\mathfrak{s}$ as in Example \ref{ex_2}
with additional condition that
$\psi(\mathfrak{s})(\mathbb{R}^n)=\mathbb{R}^n.$
\end{remark}

\newpage

\section{One class of solvmanifolds}

The following information is standard, details can be found e.g. in  \cite{Ale, NikNik, Nik2005}.

A simply connected solvable Lie group $\mathcal{S}$ together
with left-invariant Riemannian metric
$\rho$ is called a solvmanifold.
A metric Lie algebra is the pair $(\mathfrak{s},Q)$,
where $\mathfrak{s}$ is a Lie algebra, and $Q$ is
some inner product on
$\mathfrak{s}$.
A solvmanifold
$(\mathcal{S},\rho)$ determines an inner product $Q$
on the solvable Lie algebra $\mathfrak{s}$ of ${\mathcal S}$, and conversely,
any inner product $Q$ on $\mathfrak{s}$ gives rise to a unique left-invariant
metric $\rho$ on ${\mathcal S}$.

Let us choose
some orthonormal basis $\{X_i\}$ in
the metric Lie algebra $(\mathfrak{s},Q)$.
Consider the vector $H_Q$ such that
$Q(H_Q,X)=\trace (\ad (X))$ for every $X\in \mathfrak{s}$.
Therefore, $H_Q=0$ iff the Lie algebra $\mathfrak{s}$
(and the Lie group ${\mathcal S}$) is unimodular.

Let
$({\mathfrak s},Q)$ be
a solvable metric Lie algebra.
Consider its nilradical $\mathfrak{n}$, i.e. the maximal
nilpotent ideal of
$\mathfrak{s}$. It is clear that
$[\mathfrak{s},\mathfrak{s}]\subset \mathfrak{n}$ and, generally speaking,
$[\mathfrak{s},\mathfrak{s}]\neq \mathfrak{n}$. We also consider
$\mathfrak{a}$, the orthogonal
complement to
$\mathfrak{n}$ in $\mathfrak{s}$ relative to $Q$.
Note that the vector $H_Q$ determined by the equations
$Q(H_Q,Z)=\trace (\ad (Z))$ for all $Z\in \mathfrak{s}$, sits in $\mathfrak{a}$.
If $H_Q=0$, then $\mathfrak{s}$ is unimodular.
If $H_Q \neq 0$, then the subspace
$$
\mathfrak{u}=\{Z\in \mathfrak{s}\,|\, \trace (\ad (Z))=Q(H_Q,Z)=0 \}
$$
is the maximal unimodular subalgebra in $\mathfrak{s}$,
in particular, $\mathfrak{n}\subset \mathfrak{u}$.

Consider any $Y \in \mathfrak{a}$, the action of this vector on $\mathfrak{n}$
is determined by the operator $T_Y=\ad (Y)|_{\mathfrak{n}}$.
If we
fix some $Q$-orthonormal basis $(e_1, ..., e_k)$ ($k=\dim(\mathfrak{n})$)
in $\mathfrak{n}$,
then the operator $T_Y$ is determined by a matrix
which we will denote by the same symbol.

Let us choose some $Q$-orthonormal  basis
$(f_1,...,f_m)$ ($m=\dim(\mathfrak{a})$)  in $\mathfrak{a}$ in such a way that
\begin{equation*}
t:=\trace (\ad (f_1))=\|H_Q\|\geq 0,  \quad  \trace (\ad(f_j))=0,  \quad  2 \leq j \leq m.
\end{equation*}
It is easy to see that $f_1=\|H_Q\|^{-1}H_Q$ for any non-inimodular $(\mathfrak{s},Q)$.
If $\mathfrak{s}$ is unimodular, we can choose $f_1 \in \mathfrak{a}$ arbitrarily (in this case $t=0$).

Now, we consider $T_i=\ad (f_i)|_{\mathfrak{n}}$, $i=1,2,\dots,m$.
It is clear that $T_i \in \operatorname{Der}(\mathfrak{n})$,
where $\operatorname{Der} (\mathfrak{n})$ is the Lie algebra
of derivations of $\mathfrak{n}$.

If the subspace $\mathfrak{a}$ is abelian, then
the metric Lie algebra $(\mathfrak{n},Q|_\mathfrak{n})$ and the operators
$T_i$ ($1\leq i \leq m$)
uniquely determine the metric Lie algebra
$(\mathfrak{s}, Q)$.
A representation of this type we will write as follows:
$$
(\mathfrak{s}, Q)=\mathfrak{s}(T_1,...,T_m; \mathfrak{n}).
$$
Note that in this case $\mathfrak{s}$ is
an orthogonal semidirect sum of the nilpotent algebra $\mathfrak{n}$ and commutative algebra $\mathfrak{a}$.

The most simple type of $k$-dimensional nilpotent
metric Lie algebra is the commutative Lie algebra,
which we will denote by ${\mathcal E}^k$.
\smallskip

In what follows, we deal with
solvable metric Lie algebras
$(\mathfrak{s}, Q)$ with commutative nilradicals
$\mathfrak{n}$ and with commutative orthogonal complements
$\mathfrak{a}$.
In this case, the metric Lie algebra
$(\mathfrak{n}, Q|_{\mathfrak{n}})$
is isomorphic to ${\mathcal E}^k$, where $k=\dim (\mathfrak{n})$.
If
$m=\dim(\mathfrak{a})$, then the corresponding
metric algebra of this type
can be defined as
\begin{equation}\label{eq.metabalg}
\mathfrak{s}(T_1,...,T_m; {\mathcal E}^k),
\end{equation}
where $T_i:
{\mathcal E}^k \rightarrow {\mathcal E}^k$
are pairwise commuting endomorphisms satisfying the conditions
$\trace (T_1)\geq 0$ and $\trace (T_i)=0$ when $2\leq i\leq m$.
\smallskip

It is important that all curvature characteristics of the solvmanifold $(\mathcal{S},\rho)$ can be calculated in terms of the corresponding
metric algebra $(\mathfrak{s}, Q)$.

Let $X, X^{\prime} \in \mathfrak{n}$, $Y, Y^{\prime} \in \mathfrak{a}$,
$X=\sum\limits_{i=1}^{k} a_i e_i$,
$X^{\prime}=\sum\limits_{i=1}^{k} b_i e_i$,
$A_i=\ad (f_i)|_{\mathfrak{n}}$,
$A_Y=\ad (Y)|_{\mathfrak{n}}$.

Then the following formulas are valid for the curvature \cite{Ale}:
$$
K(Y, Y^{\prime})=0, \quad K(Y,X)=-Q(A_Y^2X,X)\leq 0 \quad \mbox{if} \quad A_Y^t=A_Y,
$$
\begin{equation*}\label{Kgeneral}
K(X,X^{\prime})=
\sum\limits_{j=1}^m (Q(A_j^s X, X^{\prime}))^2-
\sum\limits_{j=1}^mQ(A_j^s X, X)Q(A_j^s X^{\prime}, X^{\prime}),
\end{equation*}
where $K(X,Y)=(R(X,Y)Y,X)$, $A_Y^t$ is the transpose of $A_Y$, $A_j^s = \frac{1}{2} (A_j^t+A_j)$ is  the
symmetric part of $A_j$, $j=1,\dots,m$.

Relative to the basis $\{e_1,...,e_k, f_1,...,f_m\}$, the matrices of the operators $\ad (f_j)$ and $\ad (e_i)$ have the forms
\begin{equation*}\label{eq:adad}
\ad (f_j) = \left(
\begin{array}{cc}
 A_j  & 0 \\
 0  & 0 \\
 \end{array}
 \right), \quad
\ad (e_i) = \left(
\begin{array}{cc}
 0  & C_i\\
 0 & 0 \\
 \end{array}
 \right),
\end{equation*}
for some $(k \times k)$-matrices $A_j$  and $(k \times m)$-matrices $C_i$. The matrix of the Ricci operator
of the solvable metric Lie algebra $(\mathfrak{s}, Q)$ has the form (see  \cite[Theorem~3]{NikNik})
\begin{equation*}\label{eq.ricci1}
\Ric = \left(
\begin{array}{rr}
 R & 0 \\
 0 & -L \\
 \end{array}
\right),
\end{equation*}
where
$$
R = \frac{1}{2} \sum\nolimits_{j=1}^{m}[A_j,A_j^t]  - \trace(A_1) A_1^s,
$$
$A_j^t$ and $A_j^s = \frac{1}{2} (A_j^t+A_j)$ are the transpose and the
symmetric part of $A_j$, respectively, $L$ is an $(m \times m)$-matrix with the entries $L_{pq}= \trace (A_p^sA_q^s)$.
\smallskip

In particular, if $m=1$ and $A_1=\diag(a_1,a_2,\dots,a_k)$, then
\begin{equation}\label{eq.ric1}
\Ric = \diag \left(-t\cdot a_1, -t\cdot a_2,\,\cdots\,, -t\cdot a_k, -\sum_{i=1}^k a_i^2\right), \quad t =\sum_{i=1}^k a_i.
\end{equation}

Recall that a metric Lie algebra $(\mathfrak{s}, Q)$ is  Einstein  if the corresponding Ricci operator $\Ric$ is a multiple of $Q$.

\begin{example}[\cite{Ale, Nik2005}]\label{ex.metab}
Let
$(\mathfrak{s}, Q)$ be a nonunimodular solvable Einstein metric algebra
with a commutative nilradical $\mathfrak{n}$
and with the commutative orthogonal
(relative to $Q$)
complement $\mathfrak{a}$,
$\Ric (Q)=-r^2 Q$ ($r>0$).
It is known that
$(\mathfrak{s}, Q)$
is isometric to a metric Lie algebra of the following type
\begin{equation}\label{Iwacom}
\mathfrak{s}\bigl(\diag ({\xi}_1),..., \diag ({\xi}_m);{\mathcal E}^k\bigr),
\end{equation}
where
${\xi}_i \in \mathbb{R}^k$ $(1\leq i \leq m)$ are such that
$\frac{1}{r}\,{\xi}_i$
form an orthonormal set
of $m$ vectors in Euclidean space
$\mathbb{R}^k$; $\diag ({\xi}_1)=\frac{r}{\sqrt{k}}\Id$;
and $\diag ({\xi}_i)$ is a matrix representation of the corresponding
endomorphism relative to some
orthonormal basis in
${\mathcal E}^k$.

For $m=1$,
the metric Lie algebra
$\mathfrak{s}(\diag ({\xi});{\mathcal E}^k)$,
where
$\diag ({\xi})=\frac{r}{\sqrt{k}}\Id$,
has constant negative curvature $K=-\frac{r^2}{k}$,
and corresponds to the Lobachevsky space,
which is modeled as the symmetric space $SO_{0}(k+1,1)/SO(k+1)$.
The Lie groups that correspond to these metric Lie algebra have integrable left-invariant distributions,
as it was discusses in the first section.

If $m\geq 2$, then
for a fixed $r>0$
there is continuous family of pairwise nonisometric
Einstein metric algebras of the type \eqref{Iwacom}.
The maximal value of sectional curvatures
of such metric Lie algebra
is nonnegative and is equal
\begin{equation*}\label{Secmax}
K_{\max}= \max\limits_{i\neq j} \left(
-\sum\limits_{l=1}^m {\xi}_l^i {\xi}_l^j \right).
\end{equation*}

If $k=m=2$ then we can get the direct metric product of two Lobachevsky planes $SO_{0}(2,1)/SO(2)$
of constant negative curvature $K=-{r^2}$ of the form \eqref{Iwacom}.

If $k=3$ and $m=2$ we get a one-parameter family of Einstein metrics of the form~\eqref{Iwacom}.
They can be described by the following operators in the basis $\{e_1,e_2,e_3\}$:
$$
\ad (f_1)=\diag \left(\frac{r}{\sqrt{3}},\frac{r}{\sqrt{3}},\frac{r}{\sqrt{3}}\right)\!,
\ad (f_2)=\diag \left(\frac{t+\sqrt{2-3t^2}}{2}r, -t\cdot r, \frac{t-\sqrt{2-3t^2}}{2}r\right)
$$
for some $t\in[0,1/\sqrt{6}]$, see details in \cite{Nik2005}.
When $t=1/\sqrt{6}$, the corresponding solvmanifold
is the product of simply connected
$2$-dimensional and $3$-dimensional spaces of constant
curvature (the product of
$SO_0(2,1)/SO(2)$ and $SO_0(3,1)/SO(3)$).
When $t\in [0,1/\sqrt{6})$, the corresponding solvmanifold
has sectional curvatures of both signs.
The maximal value of the sectional curvature is $r^2(1/6-t^2)$.
By Theorem~\ref{gordonnonpos}, these Einstein solvmanifolds are not geodesic orbit for $t\in [0,1/\sqrt{6})$
(recall that all symmetric spaces of non-compact type have nonpositive sectional curvature).
Hence, we get a non-compact examples of Einstein manifolds that have no GO-property.

The corresponding problem on existing such manifolds was posed in \cite{Nik2019}.
It should be noted that many examples of such manifolds have now been constructed in the compact case,
see details in \cite{LTX, Souris}.
In the non-compact case, all Einstein solvable manifolds that are not isometric to symmetric spaces provide such examples, according to Theorem \ref{gordonnonpos}.

If $k=4$ and $m=2$ we get a two-parameter family of Einstein metrics of the form~\eqref{Iwacom}, see details in \cite{NikNik}.
As $k$ increases, the dimension of the corresponding space of Einstein metrics of the specified type also increases.
\end{example}

\bigskip

Now we consider some special homogeneous spaces with integrable invariant distributions.
They will be constructed using the family \eqref{eq.metabalg},
where
$$
T_i=\ad(f_i)|_{\mathfrak{n}}=\diag(\xi_i)\Id=\diag(\xi_i^1,\xi_i^2,\dots,\xi_i^k)\Id, \quad 1\leq i \leq m,
$$
$\trace(T_1)=\sum_{j=1}^k \xi_1^j =t>0$, and $\trace(T_i)=\sum_{j=1}^k \xi_i^j=0$ for $2\leq i \leq m$.

Let us consider natural numbers $k_1\geq k_2 \geq \cdots \geq k_l \geq 2$ such that $k_1+k_2+\cdots +k_l=k$.
{\it Let us assume that all vectors ${\xi}_j=(\xi_j^1,\xi_j^2,\dots,\xi_j^k)\in \mathbb{R}^k$ are such that
each of them has pairwise equal coordinates for indices from $\sum_{j=1}^{q-1} k_j+1$ to $\sum_{j=1}^{q} k_j$ for all $q=1,2,\dots,l$.}

It is easy to see that the corresponding left-invariant metrics on the metabelian group $\mathcal{S}$ with the Lie algebra $\mathfrak{s}$ admit
additional symmetries. Indeed, we can consider ${\mathcal E}^k$ as a direct metric product
${\mathcal E}^{k_1}\times {\mathcal E}^{k_2}\times \cdots \times {\mathcal E}^{k_{l-1}}\times {\mathcal E}^{k_l}$.
Any left-invariant metric under consideration is invariant under the action of
$$
H=SO(k_1)\times SO(k_2) \times \cdots \times SO(k_l),
$$
where $SO(k_q)$ acts naturally on  ${\mathcal E}^{k_q}$ (that consists of vectors $a=(a^1,a^2, \dots,a^k) \in {\mathcal E}^k$ with nonzero elements only for indices
from $\sum_{j=1}^{q-1} k_j+1$ to $\sum_{j=1}^{q} k_j$, $q=1,2,\dots,l$)
and trivially on other ${\mathcal E}^{k_i}$, $i\neq q$.

In this case, we have the homogeneous space
$\mathcal{S}H/H=[(H\times \mathbb{R}^m)\ltimes \mathbb{R}^k]/H$ with a suitable invariant Riemannian metric, which is isometric to $\mathcal{S}$.
We will call such homogeneous spaces as $\mathcal{S}H/H$ with $l$ blocks.

\begin{lemma}\label{le.bblock1}
Any homogeneous space $\mathcal{S}H/H$ with $l$ blocks has integrable invariant distributions.
\end{lemma}

\begin{proof}
Let us show that the space $\mathcal{S}H/H$ has integrable invariant distributions.
It is clear that $\mathfrak{n}$ can be decomposed in the direct sum of pairwise inequivalent $\Ad(H)$-invariant
submodules $\mathfrak{n}_j$ of dimension $k_j$, $j=1,...,l$. In fact, we have $\mathfrak{n}_j={\mathcal E}^{k_j}$ for $j=1,2,\dots, l$.
Recall that  $k_j\geq 2$ for $1\leq j \leq l$ by our assumptions.
This proves that the $\Ad(H)$-invariant modules $\mathfrak{n}_j$ are pairwise non-isomorphic.

On the other hand,
any irreducible submodule of $\mathfrak{a}$ is one-dimensional ($H$ acts trivially on $\mathfrak{a}$).
Now, if we have a subspace $\mathfrak{q} \subset \mathfrak {s}$
which is invariant under the adjoint action of the Lie algebra $\mathfrak{h}=\mathfrak{so}(k_1)\oplus \mathfrak{so}(k_2) \oplus \cdots \oplus \mathfrak{so}(k_l)$
of the Lie group $H$,
then it has a form $\mathfrak{q}=\mathfrak{q}_1 \oplus \mathfrak{q}_2$, where $\mathfrak{q}_1 \subset \mathfrak{n}$ and $\mathfrak{q}_2 \subset \mathfrak{a}$
(it is important that $k_j \geq 2$ for all $j$). It is easy to see that $\mathfrak{q}_1$ is a direct sum of some submodules $\mathfrak{n}_j$, $j=1,...,l$.
Moreover, any $\mathfrak{n}_j$ is invariant under action of every operator $T_i=\diag(\xi_i)\Id$, $i=1,\dots, m$, by the above construction.
Hence, $[\mathfrak{q}_2,\mathfrak{n}_j]\subset [\mathfrak{a},\mathfrak{n}_j]\subset \mathfrak{n}_j$.
It implies that $[\mathfrak{q}_2,\mathfrak{q}_1]\subset \mathfrak{q}_1$,
$\mathfrak{q}=\mathfrak{q}_1 \oplus \mathfrak{q}_2$ is a subalgebra of the Lie algebra $\mathfrak{s}$, and $\mathfrak{h}\oplus \mathfrak{q}$ is a subalgebra of
$\mathfrak{h}\oplus \mathfrak{s}$.
\end{proof}

Now, if the solvmanifold, corresponding to a such metric Lie algebra \eqref{eq.metabalg} has negative Ricci curvature and is not isometric to a symmetric space
(it can be close to the direct product of several Lobachevsky spaces with the same Ricci curvature), then it is not geodesic orbit according to Theorem \ref{gordonnonpos}.

\begin{theorem}
Suppose that $\mathcal{S}H/H$ is constructed with using $l\geq 2$ blocks and with $m=1$,
where $A_1=\diag(a_1,a_2,\dots,a_k)$, $\trace(A_1)=\sum_{j=1}^k a_j =t>0$,
$H=SO(k_1)\times SO(k_2) \times \cdots \times SO(k_l)\subset SO(\mathfrak{n})$ and such that

{\rm{(1)}} all $a_j$, $1\leq j \leq k$, are positive;

{\rm{(2)}} the numbers  $a_j$ are pairwise equal within a given block;

{\rm{(3)}} the Ricci eigenvalues $-a_i\cdot t$  {\rm(}see \eqref{eq.ric1}{\rm)} are different in different blocks.

Then  $\mathcal{S}=\mathcal{S}H/H$ supplied with the corresponding invariant Riemannian metric is not a geodesic orbit manifold.
\end{theorem}

\begin{proof}
It is clear that the matrix $L$ (see \eqref{eq.ric1}) is non-negative.
Therefore, the Ricci operator $\Ric$ is negative definite. If $\mathcal{S}=\mathcal{S}H/H$ is a geodesic orbit manifold,
then it is a symmetric space by Theorem \ref{gordonnonpos}. It is easy to see that it should be a symmetric space of non-compact type.
Since the Ricci curvature has different value on different blocks and we have at least $2$ blocks, then it is a direct metric product of at least two
irreducible symmetric spaces of non-compact type. But in this case $m$ (the rank) should also be at least $2$. This contradiction proves the theorem.
\end{proof}

The above argument can be modified in various ways in order to construct some new examples.
%\bigskip

%\bigskip


\begin{thebibliography}{99}
\bibitem{AV}
{\sl Akhiezer~D.~N., Vinberg~\'{E}.~B.}
Weakly symmetric spaces and spherical varieties. Transform. Groups, 4 (1999), 3--24, {\bf MR}1669186, {\bf Zbl.}0916.53024.

\bibitem{Ale}
{\sl Alekseevskii~D.~V.}
Homogeneous Riemannian spaces of negative curvature.
Mat. Sb. (N.~S.), \textbf{96} (1975), 93~--~117 (in Russian).
English translation in: Math. USSR-Sb., \textbf{25} (1976), 87~--~109, {\bf MR}0362145, {\bf Zbl.}0325.53043.


\bibitem{AlNik}
{\sl Alekseevsky~D.~V., Nikonorov~Yu.~G.}
Compact Riemannian manifolds with homogeneous geodesics.
SIGMA, \textbf{5(093)} (2009), 16 pp., {\bf MR}2559668, {\bf Zbl.}1189.53047.



\bibitem{Arv17}
{\sl Arvanitoyeorgos~A.}
Homogeneous manifolds whose geodesics are orbits. Recent results and some open problems.
Irish Math. Soc. Bulletin, 79 (2017), 5--29, {\bf MR}3701182, {\bf Zbl.}1380.53056.



\bibitem{Ber1989}
{\sl Berestovskii~V.~N.}
Homogeneous spaces with an intrinsic metric. Dokl. Akad. Nauk SSSR, \textbf{301(2)} (1988), 268--271 (in Russian).
English translation: Soviet Math. Dokl., \textbf{38(1)} (1989), 60--63, {\bf MR}0967817, {\bf Zbl.}0671.53037.



\bibitem{Ber1989a}
{\sl Berestovskii~V.~N.}
Homogeneous manifolds with intrinsic metric. II. Siber. Math. J., \textbf{30(2)} (1989), 180--191, {\bf MR}0997464, {\bf Zbl.}0681.53029.



\bibitem{Ber1992}
{\sl Berestovskii~V.~N.}
Compact homogeneous manifolds with integrable invariant distributions (Russian), Izv. Vyssh. Uchebn. Zaved., Mat., \textbf{6(361)} (1992), 42--48.
English translation in:
Russ. Math.,  \textbf{36(6)}  (1992), 39--45, {\bf MR}1213333, {\bf Zbl.}0782.53040.



\bibitem{Ber1995}
{\sl Berestovskii~V.~N.}
Compact homogeneous manifolds with integrable invariant distributions, and scalar curvature. Mat. Sbornik, \textbf{186(7)} (1995), 941--950,
{\bf MR}1355453, {\bf Zbl.}0872.53033.



\bibitem{BerGor2014}
{\sl Berestovskii~V.~N., Gorbatsevich~V.~V. }
Homogeneous spaces with inner metric and with integrable invariant distributions. Analysis and Mathematical Physics, \textbf{4(4)} (2014), 263--331,
{\bf MR}3266504, {\bf Zbl.}1316.53035.


\bibitem{BerNik2020}
{\sl Berestovskii~V.~N., Nikonorov~Yu.~G.}
Riemannian Manifolds and Homogeneous Geodesics. Springer Monographs in Mathematics.
Springer Nature Switzerland AG. Cham, 2020. XXII+482 pp.

\bibitem{Cart1926}
{\sl Cartan~\'E.}
Sur une classe remaqueble d'espaces de Riemann.  I, II.
Bull. Soc. Math. de France, \textbf{54} (1926), 214--264; \textbf{55} (1927), 114--134, {\bf Zbl.}53.0390.01.


\bibitem{Gor2008}
{\sl Gorbatsevich~V.~V.}
Invariant intrinsic Finsler metrics on homogeneous spaces and strong subalgebras in Lie algebras. Siber. Math. J., \textbf{49(1)} (2008), 36--47,
{\bf MR}2400569, {\bf Zbl.}1164.53393.


\bibitem{Gord1981}
{\sl Gordon~C.}
Transitive riemannian isometry groups with nilpotent radicals. Annales de l'institut Fourier. \textbf{31(2)} (1981), 193--204,
{\bf MR}0617247, {\bf Zbl.}0441.53034.

\bibitem{Gor96}
{\sl Gordon~C.}
Homogeneous Riemannian manifolds whose geodesics are orbits,
in: Progress in Nonlinear Differential Equations. V.~20. Topics in geometry: in
memory of Joseph D'Atri. Birkh{\"a}user, 1996, 155--174, {\bf MR}1390313, {\bf Zbl.}0861.53052.


\bibitem{GorNik2018}
{\sl Gordon~C.~S., Nikonorov~Yu.~G.}
Geodesic orbit Riemannian structures on $\mathbb{R}^n$. Journal of Geometry and Physics,
\textbf{134} (2018), 235--243, {\bf MR}3886938, {\bf Zbl.}1407.53032.



\bibitem{KV}
{\sl Kowalski~O., L.~Vanhecke~L.}
Riemannian manifolds with homogeneous geodesics.
Boll. Un. Mat. Ital. B (7),  5(1)  (1991),  189--246, {\bf MR}1110676, {\bf Zbl.}0731.53046.



\bibitem{Kram}
{\sl Kr\"amer~M.}
Eine Klassification bestimmter Untergruppen kompakter zusammenh\"angender Liegruppen.
Commun. Algebra, \textbf{3(8)} (1975), 691--737,
{\bf MR}0376965, {\bf Zbl.}0309.22013.

\bibitem{LTX}
{\sl Luo~Wenyan, Tan~Ju, Xu~Na.}
New non-geodesic orbit Einstein metrics on $Sp(n)$.
Manuscr. Math. \textbf{176(1)} (2025), Paper No. 13, {\bf MR}4860308, {\bf Zbl.}1566.53052.


\bibitem{Mant1966}
{\sl Manturov~O.~V.}
Homogeneous Riemannian spaces with an irreducible rotation group (in Russian).
Trudy Sem. Vektor. Tenzor. Anal. \textbf{13} (1966), 68--145,  {\bf MR}0210031, {\bf Zbl.}0173.24102.


\bibitem{Mant1966e}
{\sl Manturov~O.~V.}
Homogeneous Riemannian spaces with irreducible rotation group.
Tensor and vector analysis, Gordon and Breach, Amsterdam, 1998, 101--192, {\bf MR}1663676, {\bf Zbl.}0970.53030.


\bibitem{Milnor1976}
{\sl Milnor~J.}
Curvatures of left invariant metrics on Lie groups. Adv. Math. 21:3(1976), 293--329,  {\bf MR}0425012, {\bf Zbl.}0341.53030.

\bibitem{NikNik}
{\sl Nikitenko~E.~V., Nikonorov~Yu.~G.}
Six-dimensional Einstein solvmanifolds. Mat. tr., \textbf{8} (2005), 71--121 (Russian).
English translation in: Siberian Adv. Math., \textbf{16(1)} (2006), 66--112, {\bf MR}2206760, {\bf Zbl.}1249.53067.


\bibitem{NikZilW25}
{\sl Nikolayevsky~Y.,  Ziller~W.}
Non-singular geodesic orbit nilmanifolds,
Ann. Mat. Pura Appl., 205 (2026),
1939--1973,
\url{https://doi.org/10.1007/s10231-026-01661-9}.


\bibitem{Nik2005}
{\sl Nikonorov~Yu.~G.}
Noncompact homogeneous Einstein 5-manifolds. Geom. Dedicata, \textbf{113} (2005), 107--143, {\bf MR}2171301, {\bf Zbl.}1081.53041.

\bibitem{Nik2013}
{\sl Nikonorov~Yu.~G.}
Geodesic orbit manifolds and Killing fields of constant length.
Hiroshima Math. J.,  \textbf{43(1)} (2013), 129--137,  {\bf MR}3066528, {\bf Zbl.}1276.53046.



\bibitem{Nik2017}
{\sl Nikonorov~Yu.~G.}
On the structure of geodesic orbit Riemannian spaces.
Ann. Glob. Anal. Geom., \textbf{52} (2017), 289--311,  {\bf MR}3711602, {\bf Zbl.}1381.53088.

\bibitem{Nik2019}
{\sl Nikonorov~Yu.~G.}
On left-invariant Einstein Riemannian metrics that are not geodesic orbit.
Transform. Groups, \textbf{24(2)}  (2019), 511--530, {\bf MR}3948943, {\bf Zbl.}1427.53064.

\bibitem{Selberg}
{\sl Selberg~A.}
Harmonic analysis and discontinuous groups in weakly symmetric Riemannian spaces, with applications to Dirichlet series. J. Indian Math. Soc., \textbf{20} (1956), 47--87,
{\bf MR}0088511, {\bf Zbl.}0072.08201.


\bibitem{Souris}
{\sl Souris~N.~P.}
Einstein Lie groups, geodesic orbit manifolds and regular Lie subgroups.
Commun. Contemp. Math.,  \textbf{27(1)} (2025), Article ID 2350068, 29 p., {\bf MR}4833382, {\bf Zbl.}1555.53094.

\bibitem{WanZil1991}
{\sl Wang~M., Ziller~W.}
On isotropy irreducible Riemannian manifolds. Acta Math., \textbf{166(3-4)} (1991), 223--261,  {\bf MR}1097024, {\bf Zbl.} 0732.53040.


\bibitem{WanZil1993}
{\sl Wang~M., Ziller~W.}
Symmetric spaces and strongly isotropy irreducible spaces. Math. Ann., \textbf{296(2)} (1993), 285--326, {\bf MR}1219904, {\bf Zbl.}0804.53075.



\bibitem{Wolf1968}
{\sl Wolf~J.~A.}
The geometry and structure of isotropy irreducible homogeneous spaces.
Acta Math., \textbf{120} (1968), 59--148. (Correction, Acta Math., \textbf{152(1-2)} (1984), 141--142,
{\bf MR}0736216, {\bf Zbl.}0157.52102.


\bibitem{Zil1996}
{\sl Ziller~W.}
Weakly symmetric spaces,
in: Progress in Nonlinear Differential Equations. V.~20. Topics in geometry: in
memory of Joseph D'Atri. Birkh{\"a}user, 1996, 355--368,
{\bf Zbl.}0860.53030, {\bf MR}1390324.



\end{thebibliography}
\end{document}